\documentclass[final,leqno,onefignum,onetabnum]{siamltex1213}
\usepackage{amsmath}
\usepackage{amsfonts}
\usepackage{color}

\newtheorem{rem}[theorem]{Remark}

\title{Reconstruction of Lam\'{e} moduli and density
       at the boundary enabling directional elastic wavefield decomposition}

\author{Maarten V. de Hoop\thanks{Simons Chair in Computational and
    Applied Mathematics and Earth Science, Rice University, Houston,
    TX 77005, USA (\tt{mdehoop@rice.edu}).}
\and
Gen Nakamura\thanks{Department of Mathematics, Hokkaido University,
  Sapporo 060-0810, Japan (\tt{nakamuragenn@gmail.com}).}
\and
Jian Zhai\thanks{Department of Computational and Applied Mathematics,
  Rice University, Houston, TX, 77005, USA
  (\tt{jian.zhai@rice.edu}).}}

\begin{document}
\maketitle
\slugger{siap}{xxxx}{xx}{x}{x--x}%slugger should be set to mms, siap, sicomp, sicon, sidma, sima, simax, sinum, siopt, sisc, or sirev

\begin{abstract}
We consider the inverse boundary value problem for the system of
equations describing elastic waves in isotropic media on a bounded
domain in $\mathbb{R}^3$ via a finite-time Laplace transform. The data is
the dynamical Dirichlet-to-Neumann map. More precisely, using the full
symbol of the transformed Dirichlet-to-Neumann map viewed as a
semiclassical pseudodifferential operator, we give an explicit
reconstruction of both Lam\'{e} parameters and the density, as well as
their derivatives, at the boundary. We also show how this boundary reconstruction leads to a decomposition of 
incoming and outgoing waves.
\end{abstract}

\begin{keywords}inverse boundary value problem, layer stripping,
elastic waves, isotropy\end{keywords}

\begin{AMS}35R30, 35L10\end{AMS}

\pagestyle{myheadings} \thispagestyle{plain} \markboth{Reconstruction
  of Lam\'{e} parameters and density at the boundary}{Maarten
  V. de Hoop, Gen Nakamura and Jian Zhai}

\section{Introduction}${}$
\newline
\indent
We let $\Omega\subset\mathbb{R}^3$ be a bounded domain with a smooth
boundary $\partial\Omega$. We consider the following initial boundary
value problem for the system of equations describing elastic waves
\begin{equation}\label{EQ no.1}
\begin{cases}
\rho\partial^2_tu=\operatorname{div} (\mathbf{C}\varepsilon(u))=:Lu~~\text{in}~\Omega_T=\Omega\times(0,T) ,\\
u=f~~\text{on}~\Sigma=\partial\Omega\times(0,T) ,\\
u(x,0)=\partial_t u(x,0)=0~~\text{in}~\Omega,
\end{cases}
\end{equation}
with $f(x,0)=0$ and $\frac{\partial}{\partial t}f(x,0)=0$ for $x\in\partial\Omega$.
Here, $u$ denotes the displacement vector and
$\varepsilon(u)=(\varepsilon_{ij}(u))=(\nabla u+(\nabla u)^T)/2$ the
linear strain tensor which is the symmetric part of $\nabla u$. Furthermore,
$\mathbf{C}=\mathbf{C}(x)=(\dot{C}_{ijkl}(x))$ is the elasticity
tensor and $\rho$ is the density of mass. We assume that $\mathbf{C}$ is
isotropic, that is,
\begin{equation}\label{Cartesian tensor}
\dot{C}_{ijkl}(x)=\lambda(x) \delta_{ij}\delta_{kl}+\mu(x)(\delta_{ik}\delta_{jl}+\delta_{il}\delta_{jk})
\end{equation}
with Kronecker's delta $\delta_{ij}$ and Lam\'e moduli $\lambda,\,\mu\in C^\infty(\overline\Omega)$ such that $\mu>0$ and $\lambda+2\mu>0$ on $\overline\Omega$. Also, $\rho\in C^\infty(\overline{\Omega})$ and $\rho>0$ on $\overline{\Omega}$.

The hyperbolic or dynamical Dirichlet-to-Neumann map (DN map)
$\Lambda_T$ is defined according to
\begin{equation}\label{Lambda_T}
\Lambda_T:H^2(\Sigma)\ni f\mapsto\partial_L u:=(\mathbb{C}\varepsilon(u))\nu|_{\partial\Omega}\in C([0,T],H^{1/2}(\partial\Omega)),
\end{equation}
where $u$ is the solution of (\ref{EQ no.1}),
$\mathbb{C}\varepsilon(u)$ is a $3\times3$ matrix with its $(i,k)$
component $(\mathbb{C}\varepsilon(u))_{ik}$ given by
$(\mathbb{C}\varepsilon(u)) _{ik}=\sum_{j,l=1}^3
\dot{C}_{ijkl}\varepsilon_{kl}(u)$, $\nu$ is the outward unit normal
to $\partial\Omega$. Physically, $\partial_L u$ denotes the traction at $\partial\Omega$. 

In this paper, we consider
the inverse problem of recovering $\lambda,\mu,\rho$, as well as all
their derivatives, at the boundary $\partial\Omega$ from $\Lambda_T$. Our
major result for this inverse problem is as follows.

\medskip

\begin{theorem} \label{main thm}
The DN map $\Lambda_T$ identifies $\lambda, \mu, \rho$ and all their
derivatives on $\partial\Omega$ uniquely. There is an explicit
reconstruction procedure for these identification.
\end{theorem}\\

\begin{rem}\label{rem12}
Since the procedure we present to recover $\lambda,\mu,\rho$ and their
derivatives at $\partial\Omega$ is local, we also have a localized
version of Theorem \ref{main thm} with partial boundary data. That is,
in the definition of $\Lambda_T$ we can replace
$(\mathbb{C}\varepsilon(u)) \nu|_{\partial\Omega}$ by
$(\mathbb{C}\varepsilon(u)) \nu|_{\Gamma_0}$ and confine the boundary
sources, $f$, to those with $\operatorname{supp} \, f(.,t) \subset
\overline{\Gamma_0}\ (t \in (0,T))$, where $\Gamma_0$ is a
relatively open subset of $\partial\Omega$. We can recover $\lambda,
\mu, \rho$ and all their derivatives at $\Gamma_0$. As an additional explanation for this which should 
be given later, see the last paragraph of Section 2.
\end{rem}

\medskip

The uniqueness of the inverse problem considered here was
established by Rachele \cite{Rachele}, but giving a procedure to
recover the parameters, $(\lambda, \mu, \rho)$, has been left open for
over 15 years. One of the complications is the occurrence of two metrics in
the dynamical system of equations that cannot be straightforwardly
separated at the boundary. Indeed, the usual special solutions
including high-frequency asymptotic ones or progressive wave solutions
based on polarization decoupling are coupled at the boundary.
 
\medskip
Moreover, the determination of $(\lambda, \mu, \rho)$ on boundary implies the following

\medskip

\begin{corollary}
In addition to the conditions appearing in Theorem~\ref{main thm}, let
$\lambda, \mu, \rho$ be real analytic in the neighborhood of
$\overline\Omega$. Then $\Lambda_T$ determines uniquely $\lambda, \mu,
\rho$ on $\overline\Omega$.
\end{corollary}

\medskip
A brief remark for this corollary should be given. Although there is a more general result by Rachele \cite{Rachele},\cite{Rachele2},
the context and argument of deriving this corollary differs from those of Rachele's.

\medskip
For the static elastic inverse boundary value problem, an
explicit reconstruction of $\lambda$ and $\mu$ at the boundary from
the full symbol of the static DN map was obtained \cite{NU3, NU}. For
the reconstruction of a transversely isotropic elasticity tensor, see
\cite{NTU}. The approach was originally developed by Sylvester and
Uhlmann \cite{SU} for the electrical impedance tomography problem. The approach is also applied to Maxwell's equations \cite{S, M}. We
generalize this type of reconstruction to dynamical elastic inverse
boundary value problems.  We note that our procedure is quite general and
can be extended from isotropy to anisotropy with certain
symmetries, which is the subject of a forthcoming paper.

The key component of the reconstruction is a connection between
$\Lambda_T$ and the asymptotic expansion of the DN map, a semiclassical pseudodifferential operator $\Lambda^h$ say,
for some elliptic system of equations containing a small parameter
$h$ via a finite-time Laplace transform. M. Ikehata has been using the finite Laplace transform effectively to develop his enclosure method both for parabolic equations and hyperbolic equations, see \cite{Ikehata} and reference therein. For the convenience of our description, the partial differential operator of this system is referred by $\mathcal{M}$. We will identify $(\lambda,
\mu, \rho)$ and all of their derivatives from $\Lambda^h$ by factorizing $\mathcal{M}$
into the product of two first order semiclassical pseudodifferential operators with small parameter $h=\frac{1}{\tau}$, where $\tau$ is nothing but the
Laplace variable of this transform. Also this factorization is nothing but the one used to provide the up/down going decomposition of waves which is equivalent to the incoming/downgoing decomposition of waves in the Laplace domain. Further this decomposition can be linked to the corresponding decomposition in the space time domain. We will briefly discuss about this connection more precisely at the end of this paper. The up/down going decomposition for scalar waves is discussed in \cite{Stolk}, and for elastic waves in \cite{dHdH}.  In this paper, we connect the up/down going decomposition with Dirichlet-to-Neumann map.

\begin{figure}[h]
\begin{center}
\includegraphics[width=2.5 in]{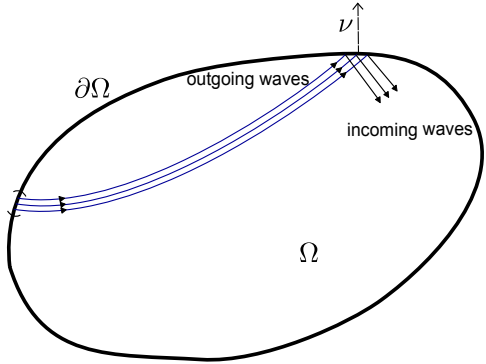}
\caption{incoming/outgoing waves.}
\end{center}
\end{figure}

Concerning our inverse problem, there are two byproducts of the factorization. The one is the explicit form of the principal symbol of $\Lambda^h(s)$ and relation of its $s$-derivatives to the non-principal symbols of $\Lambda^h(s)$. Here $\Lambda^h$ is the DN map defined likewise $\Lambda^h$ on the boundary $\Gamma(s)$ of the subdomain $\Omega(s)=\{x\in\Omega:\text{dist}(x,\partial\Omega)>s\}$  of $\Omega$ with $0<s\ll 1$. The other is that a modification  $\hat{\Lambda}(s)$ (cf. \eqref{modified DNmap}) of $\Lambda^h(s)$ satisfies a Riccati type equation. By solving the Riccati type equation, which is an initial value problem, one can propagate the data into at least a thin layer near the boundary. This technique is known as invariant embedding in extensive geophysics literature \cite{BW, CDK1, CDK, dHdH, MF}.

Knowing all the derivatives of the coefficients at $\partial\Omega$ and using the Riccati equation for $\hat{\Lambda}(s)$, we
can generate an approximation of $\Lambda_T(s)$ on 
$\Gamma(s)$. Then,
we can get an approximation of $\lambda, \mu, \rho$ and their derivatives at
$\Gamma(s)$. Repeating this process, leads to an approximation for
$\lambda, \mu, \rho$ layer by layer in the interior of $\Omega$, using
the DN map $\Lambda_T$ as the data. The associated algorithm is called
layer stripping. The layer stripping was first developed for the
electrical impedance tomography problem in \cite{CI,SCII}. Nakamura,
Tanuma and Uhlmann \cite{NTU} developed such an algorithm for the
static elastic inverse boundary value problem in the case of
transverse isotropy.

The key application of the problem we are considering is (reflection) seismology. In actual seismic acquisition, raw vibroseis data are modeled by the local Neumann-to-Dirichlet (ND) map: The boundary values are given by the normal traction underneath the base plate of a vibroseis and are zero (€˜free surface) elsewhere, while the particle displacement (in fact, velocity) is measured by geophones located in a subset of the boundary (Earth's surface) (see \cite{B}). Although, the local dynamical Dirichlet-to-Neumann map and Neumann-to-Dirichlet map do not have the same information, the transformed ND map and transformed DN map are microlocally inverse to each other. Since we are only dealing with the symbol of transformed DN map, the results of this paper apply to the practical setting.

The remainder of this paper is organized as follows. In Section 2, we
give an asymptotic identity which connects $\Lambda_T$ with $\Lambda^h$
via the finite-time Laplace transform introducing variable
$\tau=\frac{1}{h}$. Based on this identity, we only have to find a reconstruction
procedure using $\Lambda^h$. In a similar way, DN maps $\Lambda^h(s)$ are
defined on each $\Gamma(s)$. In Section 3, the full symbol of semiclassical
pseudodifferential operator $\Lambda^h(s)$ 
is analyzed using the factorization of $\mathcal{M}$. Also, as a byproduct of the factorization, we develop a layer stripping algorithm. Section 4 is devoted to giving a procedure and formulas for the
reconstruction of $(\lambda, \mu, \rho)$ and all their derivatives,
from the explicit form of the principal symbol of $\Lambda^h(s)$ at each
$\Gamma(s)$ for $0 \le s \ll 1$ in terms of the boundary normal
coordinates associated with $\Gamma(s)$. In the final section, Section 5, we will discuss about an another implication of the factorization. That is we give the aforementioned link between the outgoing/incoming decomposition of waves in the Laplace domain and that of in the space time domain.

\section{Reduction to an elliptic boundary value problem with a small parameter}${}$
\newline
\indent
First we introduce a family of symbol classes for semiclassical pseudodifferential operators.
Let $A(\cdot,\cdot;\cdot):\mathbb{R}^{2n}\times(0,h_0)\rightarrow \mathbb{C}^{\tilde{q}\times \tilde{q}}$ be a function that is smooth in $(x,\xi)\in\mathbb{R}^{2n}$ depending on $h\in(0,h_0]$ with a small $h_0>0$. We say that for $m\in\mathbb{R}$, $A$ belongs to a symbol class $\mathcal{S}(m)$, if for any $\alpha,\,\beta\in\mathbb{Z}_+^n$, there exists a
constant $C_{\alpha,\beta}>0$ such that
\[|D_x^\alpha D_{\xi}^\beta A(x,\xi;h)|\leq C_{\alpha,\beta}\langle \xi\rangle^m,\,\,(x,\xi)\in\mathbb{R}^{2n},\,h\in(0,h_0],\]
where $\mathbb{Z}_+=\mathbb{N}\cup\{0\}$, $\langle\xi\rangle=\sqrt{1+|\xi|^2}$.

We say that $A\in \mathcal{S}(m)$ is called a classical symbol if for any $\alpha\in\mathbb{Z}_+^n$ and $N=1,2,\cdots$, there exist a constant $C_{\alpha,N}>0$ such that
\[|\partial^\alpha(A-\sum_{j=0}^{N-1} h^j A_j)|\leq C_{\alpha,N}h ^N\langle\xi\rangle^m,\,\,(x,\xi)\in\mathbb{R}^{2n}\]
with each $A_j\in\mathcal{S}(m)$ independent of $h$, and write
\[A\sim\sum_{j=0}^\infty h^j A_j~~\mathrm{mod}~\mathcal{O}(h^\infty\mathcal{S}(m)).\]
$A_0$ and $\sum_{j=0}^{\infty} h^j A_j$ are called the the principal symbol $\sigma(A)$ and full symbol $\tilde{\sigma}(A)$ of $A$, respectively. In this paper we only consider classical symbols. We also denote for $A,\,A'\in\mathcal{S}(m)$,
\[A\sim A'~~\mathrm{mod}~\mathcal{O}(h^k\mathcal{S}(m)),\]
if $A-A'=h^k B$ with $B\in\mathcal{S}(m)$.

For this family of symbol classes $\mathcal{S}(m)$ with $m\in\mathbb{R}$, we use the standard theory of semiclassical pseudodifferential operators (see \cite{Martinez},\cite{Zworski} ). Also we use the above notations and terminologies for symbol classes associated to
semiclassical pseudodifferential operators on compact manifolds. In the sequel $\tilde{Q}$ will be either $\tilde{q}=3$ or $\tilde{q}=6$ which can be easily noticed in contexts. Furthermore, we abuse the notation $\mathcal{O}(h^k\mathcal{S}(m))$ to use it also for the associated
semiclassical pseudodifferential operators.

\medskip
Now we consider the following boundary value problem:
\begin{equation}\label{transformed eq}
\begin{cases}
\rho v - h^2\operatorname{div} (\mathbb{C}\varepsilon(v))=0~~\text{in}~\Omega,\\
v=\varphi~~\text{on}~\partial\Omega
\end{cases}
\end{equation}
with a small (real-valued) parameter $h \in (0,h_0]$. We define the corresponding  DN map for (\ref{transformed eq}) according to
\[\Lambda^h: H^{5/2}(\partial\Omega)\ni \varphi\mapsto h\partial_L v=h(\mathbb{C}\varepsilon(v))\nu|_{\partial\Omega}\in H^{3/2}(\partial\Omega),\]
where $v$ solves (\ref{transformed eq}). In a likewise fashion, we
define $\Lambda^h(s)$ by replacing $\Omega$ by $\Omega(s)$ and
$\partial\Omega$ by $\Gamma(s)$ while replacing $\varphi$ by $\psi \in
H^{5/2}(\Gamma(s))$ (cf. \eqref{eq.1 tensorial}) emphasizing the $s$ dependence. Of course, $\Lambda^h
= \Lambda^h(0)$.
We note that
$\Lambda^h$ and $\Lambda^h(s)$ are semiclassical pseudodifferential operators belonging to the class $\mathcal{S}(1)$ with $n=2$. 

We show that we can obtain the full symbol of $\Lambda^h$ from
$\Lambda_T$ via a finite-time Laplace transform. First we introduce the finite-time Laplace transform $w\in H^2(\Omega)$ of $u\in C([0,T],H^2(\Omega))$ by
\[w(x,\tau)=(\mathcal{L}_Tu)(x,\tau)=\int_{0}^{T}u(x,t)e^{-\tau t}\mathrm{d}t\,\,\mbox{\rm with}\,\,\tau>0.\]
In order to establish the connection between $\Lambda_T$ and $\Lambda$, we let
$\chi(t)=t^2\,(t\in[0,T])$ and define
\[\tilde{\Lambda}_T:H^{3/2}(\partial\Omega)\rightarrow W((0,T);\partial\Omega)\]
by
\[\tilde{\Lambda}_T\phi=\Lambda_T(\chi \phi).\]

For any $\phi\in H^{5/2}(\partial\Omega)$, let $u$ solve (\ref{EQ no.1}) with boundary value $f=\chi \phi$. By the estimates for solutions of hyperbolic system \eqref{EQ no.1}, we have 
\[\|\partial_t^j u(\cdot,T)\|_{H^{2-j}(\Omega)}=\mathcal{O}(\Vert\chi\phi\Vert_{H^{2}(\Sigma)})=\mathcal{O}(\Vert\phi\Vert_{H^{5/2}(\partial\Omega)})\]
for $j=0,1$ (see \cite{Wloka}). Because
\[\mathcal{L}_T(\partial_t^2 u)=\tau^2\mathcal{L}_T(u)+\partial_tu(T)e^{-\tau T}+\tau u(T)e^{-\tau T}-\partial_tu(0)-\tau u(0),\]
we have that by applying $\mathcal{L}_T$ to both sides of (\ref{EQ no.1}), and divide by $\tau^2=\frac{1}{h^2}$
\begin{equation}
\begin{cases}
\label{new eq}\rho w-h^2\text{div} (\mathbb{C}\varepsilon(w))=r~~\text{in}~\Omega,\\
w=\mathcal{L}_T(\chi \phi)~~\text{on}~\partial\Omega,
\end{cases}
\end{equation}
where $w=\mathcal{L}_T(u)$ and $r$ has an estimate $\|r\|_{H^1(\Omega)}\leq Ce^{-\kappa\tau T}\|\phi\|_{H^{5/2}(\partial\Omega)}$ with some constant $C$ for any given $\kappa$ satisfying $0<\kappa<1$.

Subtracting $(\ref{transformed eq})$ from (\ref{new eq}) with
$\varphi=\mathcal{L}_T(\chi\phi)$, we find that $z=w-v$ satisfies
\begin{equation}
\begin{cases}
\label{new eq1}\rho z-h^2\operatorname{div} (\mathbb{C}\varepsilon(z))=r~~\text{in}~\Omega ,\\
v=0~~\text{on}~\partial\Omega .
\end{cases}
\end{equation}
Hence, we have
\[\mathcal{L}_T\left(\partial_Lu\right)=\partial_Lw=\partial_Lv+\partial_L z,\]
with
\[
   \left\| h\partial_L z \right\|_{H^{3/2}(\partial\Omega)}
   \leq C\|r\|_{H^1(\Omega)}\leq Ce^{-\kappa\tau T}\|\phi\|_{H^{5/2}(\partial\Omega)},\]
by standard elliptic regularity theory. We observe that 
\[h\partial_L v=h\Lambda^h(\mathcal{L}_T(\chi\phi))\]
and
\[h\mathcal{L}_T\left(\partial_L u\right)=h\mathcal{L}_T\Lambda_T(\chi\phi).\]
Thus, defining $\tilde{\mathcal{L}}_T:H^{5/2}(\partial\Omega)\rightarrow H^{5/2}(\partial\Omega)$ by $\tilde{\mathcal{L}}_T(\phi)=\mathcal{L}_T(\chi \phi)$, we can rewrite the formula above as
\[h\mathcal{L}_T\tilde{\Lambda}_T=\Lambda^h\tilde{\mathcal{L}}_T+\mathcal{O}(e^{-\kappa\tau T}).\]
Here, $\mathcal{O}(e^{-\kappa\tau T})$ denotes an operator from $H^{5/2}(\partial\Omega)$ to $H^{3/2}(\partial\Omega)$ with the estimate
\[\|\mathcal{O}(e^{-\kappa\tau T})\|_{H^{5/2}(\partial\Omega)\rightarrow H^{3/2}(\partial\Omega)}\leq Ce^{-\kappa\tau T}.\]
We note that $\tilde{\mathcal{L}}_T$ is just a multiplication by
$\int_0^T t^2 e^{-\tau t}\mathrm{d}t$, hence it is invertible and $\tilde{\mathcal{L}}_T^{-1}$ can be
estimated by $\mathcal{O}(\tau^{-3})$ for $\tau\gg 1$. Therefore,
\[h\mathcal{L}_T\tilde{\Lambda}_T\tilde{\mathcal{L}}_T^{-1}\sim\Lambda^h\]
modulo an operator in $H^{5/2}(\partial\Omega)\rightarrow H^{3/2}(\partial\Omega)$ with the estimate $\mathcal{O}(h^\infty)$.
Therefore, we can obtain the full symbol of $\Lambda^h$ from
$\mathcal{L}_T\tilde{\Lambda}_T\tilde{\mathcal{L}}_T^{-1}$, due to what we will mention in the last paragraph of this section.

We note that in the above one can choose any smooth function for
$\chi$ that is consistent with the initial conditions such that
$\int_{0}^T \chi(t) e^{-\tau t} \mathrm{d}t$ behaves polynomially in
$\tau$.

The full symbol of a semiclassical pseudodifferential operator can be evaluated by applying to locally supported rapidly oscillating functions \cite{Zworski}.
So the analysis can be local, and thus we have Remark \ref{rem12}.

\section{Analysis of the symbol of $\Lambda^h(s)$}${}$
\newline
\indent
Given a boundary point $p_0 \in \Gamma(s)$, for any $x\in\Omega(s)$
near $p_0$, we use the boundary normal coordinates
$x=(x^1(p),x^2(p),x^3)=(y^1, y^2,x^3)=(y',x^3)$, where $p\in\Gamma(s)$
is the nearest point to $x$, $x^3=\operatorname{dist}(x,p)$, and
$(x^1(p),x^2(p))$ are the local coordinates of $\Gamma(s)$ near
$p_0$. Then $\Gamma(s)$ is locally given as $x^3=s$. Let
$(\xi_1,\xi_2,\xi_3)$, $(\eta_1,\eta_2,\eta_3)$ be conormal vectors
with respect to the coordinates $(x_1,x_2,x_3)$, $(x^1,x^2,x^3)$ such
that $\sum_{j=1}^3\xi_j \mathrm{d}x_j=\sum_{j=1}^3\eta_j
\mathrm{d}x^j$. We will use the notation $\eta =
(\eta_1,\eta_2,\eta_3)=(\eta',\eta_3)$. In boundary normal coordinates,
equation (\ref{transformed eq}) attains the form
\begin{equation}\label{eq.1 tensorial}
\begin{cases}
(\mathcal{M}v)^i=\rho g^{ik}v_k-h^2\displaystyle\sum_{j,k,l=1}^3\nabla_j (C^{ijkl}\varepsilon_{kl}(v))=0~\text{in}~\{x^3>s\}~\text{for}~1\leq i\leq 3,\\
v^i|_{x^3=s}=\psi^i,~~1\leq i\leq 3.
\end{cases}
\end{equation}
where $\nabla_j$ is the covariant derivative with respect to $\frac{\partial}{\partial x^j}$ and  $\varepsilon_{kl}(v)=2^{-1}(\nabla_l v_k+\nabla_k v_l)$ is the linear strain tensor, 
\[C^{ijkl}(x)=\sum_{a,b,c,d=1}^3 \frac{\partial x^i}{\partial x_a}\frac{\partial x^j}{\partial x_b}\frac{\partial x^k}{\partial x_c}\frac{\partial x^l}{\partial x_d}\dot{C}_{abcd}(x),\]
with $\dot{C}_{abcd}(x)$ given by \eqref{Cartesian tensor}. The induced metric $G(x)= (g^{ai}(x))$ is given by
\[g^{ai}(x)=\sum_{r=1}^3\frac{\partial x^a}{\partial x_r}(x)\frac{\partial x^i}{\partial x_r}(x).\]
In terms of Jacobi matrix $J = (\partial x^a/\partial x_r;\, 1\le a,\,r\le 3)$, $G$ takes the form $G = J J^T$.

The expression for $\Lambda^h(s)$ in boundary normal coordinates is
\begin{equation}
{\large(}\Lambda^h(s)\psi{\large)}^i=-h\displaystyle\sum_{k,l=1}^3C^{i3kl}\varepsilon_{kl}(v),\,\,1\le i\le 3
\end{equation}
at $x^3=s$. The full symbol of $\Lambda^h(s)$ can be expanded as
\[\tilde{\sigma}(\Lambda^h(s))(y',\eta')\sim \sum_{j\leq 0} h^{-j}\lambda_{-j}(s)(y',\eta')~~\mathrm{mod}~\mathcal{O}(h^\infty\mathcal{S}(1)),\]
where each $\lambda_j(s)(y',\eta')\in \mathcal{S}(1)$.

Now, we define
\begin{equation}\label{QRD}
\begin{split}
Q(x,\eta') &= \left(\sum_{j,l=1}^2C^{ijkl}(x)\eta_j\eta_l;~1\le i,\, k\,\le 3 \right) ,\\
R(x,\eta') &= \left(\sum_{j=1}^2C^{ijk3}(x)\eta_j;~1\le i,\, k\,\le 3\right) ,\\
D(x) &=\left(C^{i3k3}(x);~1\le i,\, k\,\le 3 \right) .
\end{split}
\end{equation}
The principal symbol, $M(x,\eta)$, of $\mathcal{M}$ is then given
by
\begin{equation}\label{mathcal M}
M(x,\eta)=D(x)\eta_3^2+(R(x,\eta')+R^T(x,\eta'))\eta_3+Q(x,\eta')+\rho(x) G(x).
\end{equation}

By the assumption, $M(x,\eta)$ is a positive definite matrix for
$x \in \overline{\Omega}$, $\eta \in \mathbb{R}^3 \backslash
0$.  Hence, for fixed $(x,\eta')$,
$\det D^{-1/2}M(x,\eta)D^{-1/2} = 0$
 in $\eta_3$
admits $3$ roots $\eta_3 = \zeta_j~(j=1,2,3)$ with positive imaginary
parts and $3$ roots $\overline{\zeta_j}~(j=1,2,3)$ with negative
imaginary parts. Thus,\\

\begin{lemma}[\cite{GLR}]
There is a unique factorization
\[\tilde{M}(x,\eta)=D(x)^{-1/2} M(x,\eta) D(x)^{-1/2}
=(\eta_3-\tilde{S}_0^*(x,\eta'))(\eta_3-\tilde{S}_0(x,\eta')),\]
with $\operatorname{Spec}(\tilde{S}_0(x,\eta')) \subset \mathbb{C}_+$, where $\operatorname{Spec}(\tilde{S}_0(x,\eta'))$ is the spectrum of $\tilde{S}_0(x,\eta')$.
In the above,
\[\tilde{S}_0(x,\eta'):=\left(\oint_\gamma\zeta \tilde{M}(x,\eta',\zeta)^{-1}\mathrm{d}\zeta\right)\left(\oint_\gamma \tilde{M}(x,\eta',\zeta)^{-1}\mathrm{d}\zeta\right)^{-1},\]
where $\gamma\subset\mathbb{C}_+:=\{\zeta\in\mathbb{C}:\Im{\zeta}:=\text{\rm imaginary part of $\zeta$}>0\}$ is a continuous curve enclosing all the $\zeta_j~(j=1,2,3)$.
\end{lemma}\\

Then we have the following factorization of $M(x,\eta)$:
\begin{equation}\label{fac M}
M(x,\eta)=(\eta_3-S_0^*(x,\eta'))D(x)(\eta_3-S_0(x,\eta')),
\end{equation}
where
\[S_0(x,\eta') = D^{-1/2}(x)\tilde{S}_0(x,\eta')D^{1/2}(x).\]
We arrive at

\medskip
\begin{lemma}\label{lemma1}
The operator $\mathcal{M}$ admits a factorization
\begin{equation}\label{decomp}
\begin{array}{l}
\mathcal{M}=(hD_s-S^*(x,hD_{y'};h)\\
\quad\quad\quad\quad\quad\quad+hK(x,hD_{y'}))D(x)\left(hD_s-S(x,hD_{y'};h)\right),
\end{array}
\end{equation}
where $S(x,\eta';h)\in\mathcal{S}(1)$, $K(x,\eta';h)\in\mathcal{S}(0)$.
Moreover, $hD_{y'}=(hD_{y^1},hD_{y^2})$,
$hD_{y^j}=-{\rm i}\,h\partial/\partial y^j$ $(j=1,2)$ and the principal
symbol, $S_0(x,\eta')$, of $S$ satisfies
\begin{equation}\label{spec}
\operatorname{Spec}(S_0(x,\eta'))\subset\mathbb{C}_+ .
\end{equation}
\end{lemma}
\begin{proof}
Following \eqref{mathcal M}, we write the full symbol $\tilde{\sigma}(\mathcal{M})$ of $\mathcal{M}$ in the form,
\begin{multline}\label{full1}
\tilde{\sigma}(\mathcal{M})=D(x)\eta_3^2+(R(x,\eta')+R^T(x,\eta'))\eta_3+
\\
\qquad\qquad\qquad Q(x,\eta')+\rho(x)G(x)+hF_0(x)\eta_3+hF_1(x,\eta'),
\end{multline}
where $F_1(x,\eta')\in \mathcal{S}(1)$, and $F_0(x)$ is a matrix multiplication. We
expand
\[(h D_s-S^*(x,hD_{y'};h)+hK(x,hD_{y'};h))D(x)(h D_s-S(x,hD_{y'};h)),\] yielding
\begin{multline}\label{full2}
h(D_s D)(hD_s)-S^*D (hD_s)+hKD(hD_s)-h(D_s D)S+S^*DS\\
-hKDS-hD(D_sS)-DS(hD_s)+D(hD_s)^2.
\end{multline}
Comparing (\ref{full1}) and (\ref{full2}), we find that $S$ and $K$
should satisfy
\begin{multline}\label{comp1}
-S^*(x,hD_{y'};h)D(x)+hK(x,hD_{y'};h)D(x)-D(x)S(x,hD_{y'};h)+hD_sD(x)
\\
= R(x,hD_{y'})+R^T(x,hD_{y'})+hF_0(x)
\end{multline}
and
\begin{multline}\label{comp2}
-h(D_sD(x))S(x,hD_{y'};h)+S^*(x,hD_{y'};h)D(x)S(x,hD_{y'};h)\\-hK(x,hD_{y'};h)D(x)S(x,hD_{y'};h)
-hD(x)(D_sS(x,hD_{y'};h))\\
=hF_1(x,hD_{y'})+Q(x,hD_{y'})+\rho G(x) .
\end{multline}
Eliminating $K$ in (\ref{comp2}) by using (\ref{comp1}), we get
\begin{equation}\label{Riccati for S}
(hD_s)S+S^2+D^{-1}(R+R^T+hF_0)S+hD^{-1}F_1+D^{-1}Q+D^{-1}\rho G=0.
\end{equation}
By the composition formula for symbols of pseudodifferential operators, we have
\begin{multline}\label{full expansion}
\sum_{\alpha\geq 0}\frac{{\rm i}^{|\alpha|}}{\alpha!}h^{|\alpha|}D_{\eta'}^\alpha S(x,\eta';h)D_{y'}^\alpha S(x,\eta';h)\\
+\sum_{\alpha\geq 0}\frac{{\rm i}^{|\alpha|}}{\alpha!}h^{|\alpha|}D^\alpha_{\eta'}(D^{-1}(x)(R(x,\eta')+R^T(x,\eta'))D_{y'}^\alpha S(x,\eta';h)\\+hD^{-1}(x)F_0(x)S(x,\eta';h)
+hD^{-1}(x)F_1(x,\eta')+D^{-1}(x)Q(x,\eta')\\[0.25cm]
+D^{-1}(x)\rho G+hD_sS(x,\eta';h)=0 .
\end{multline}
We introduce the expansions
\[S(x,\eta';h)\sim\sum_{j\leq 0}h^{-j}S_j(x,\eta') ,\]
\[K(x,\eta';h)\sim\sum_{j\leq 0}h^{-j}K_j(x,\eta') \]
with $S_j\in S(j+1)$, $K_j\in \mathcal{S}(j)$ for every $j$.
We construct $S$ via arranging terms of the same degree of $h$ in (\ref{full expansion}). The terms of order $\mathcal{O}(h^0)$ give
\begin{multline}\label{order 2}
D^{-1}(x)(R(x,\eta')+R^T(x,\eta'))S_0(x,\eta')+D^{-1}(x)Q(x,\eta')\\
+D^{-1}(x)\rho(x)G+S_0^2(x,\eta')=0 .
\end{multline}
Indeed, $S_0$ (cf.~(\ref{fac M})) satisfies this equation. The terms
of order $\mathcal{O}(h)$ give
\begin{multline}\label{order 1}
S_0S_{-1}+S_{-1}S_0+D^{-1}(R+R^T)S_{-1}+\sum_{|\alpha|=1}{\rm i}D_{\eta'}^\alpha S_0D_{y'}^\alpha S_0\\
+\sum_{|\alpha|=1}{\rm i}D_{\eta'}^\alpha(D^{-1}(R+R^T))D_{y'}^\alpha S_0+D^{-1}F_0S_0+D^{-1}F_1+D_sS_0=0 .
\end{multline}
The terms which are of homogeneity of order $\mathcal{O}(h^{-j})$ for $j\leq -2$ yield
\begin{multline}\label{order j}
S_0S_{j}+S_{j}S_0+D^{-1}(R+R^T)S_{j}+\sum_{l+m=j+|\alpha|\atop |\alpha|\geq 1}\frac{{\rm i}^{|\alpha|}}{\alpha!}D_{\eta'}^\alpha S_lD_{y'}^\alpha S_m\\
+\sum_{l+m=j\atop m,l\leq 0}S_lS_m+\sum_{|\alpha|=1}{\rm i}D_{\eta'}^\alpha(D^{-1}(R+R^T))D_{y'}^\alpha S_{j+1}+D^{-1}F_0S_{j+1}+D_sS_{j+1}=0.
\end{multline}
To confirm that (\ref{order 1}) and (\ref{order j}) have solutions, we note that
\[S_jS_0+S_0S_j+D^{-1}(R+R^T)S_j=-D^{-1}(Q+\rho G)S_0^{-1}S_j+S_jS_0,\]
using (\ref{order 2}). Since
\[\operatorname{Spec}(S_0)\subset\mathbb{C}_+,~~~\operatorname{Spec}(-D^{-1}(Q+\rho G)S_0^{-1})\subset\mathbb{C}_+,\]
we can indeed solve for $S_j\,(j\leq -1)$ in (\ref{order 1}) and (\ref{order j}).

After constructing the full symbol of $S$, we determine the full
symbol of $K$ from (\ref{comp1}).
\end{proof}\\

\begin{proposition}
Let $\lambda_0(s)(y,\eta')$ be the principal symbol of
$\Lambda^h(s)$. Then
\begin{equation}\label{princ}
   \lambda_0(s)(y',\eta')
   = -{\rm i} (D(x) S_0(x,\eta') + R^T(x,\eta'))|_{x^3=s} .
\end{equation}
\end{proposition}

\begin{proof}
%Let $u$ solve equation (\ref{eq.1}) and make the following observation.
%That is, in terms of the Cartesian coordinates $(x_1,x_2,x_3)$ such that $\Omega(s)$ is
%locally given by $x_3>s$, we have
%\[
%\Lambda(s)(\psi)=\lambda(\nabla\cdot u)e_3+\mu(\nabla u+(\nabla u)^T)e_3\]
%and by calculation it becomes 
%\begin{equation}\label{expression of Lambda(s)}
%\begin{split}\Lambda(s)(\psi)&=-\left(\begin{array}{c}
%\mu\partial_s u_1+\mu\partial_1u_3\\
%\mu\partial_s u_2+\mu\partial_2u_3\\
%\lambda(\partial_1u_1+\partial_2u_2+\partial_3u_3)+2\mu\partial_su_3
%\end{array}\right)\\
%&=-\left(\begin{array}{ccc}
%\mu & &\\
%&\mu &\\
%& &\lambda+2\mu
%\end{array}\right)\left(\begin{array}{c}
%\partial_s u_1\\
%\partial_s u_2\\
%\partial_s u_3
%\end{array}\right)-\left(\begin{array}{c}
%\mu\partial_1u_3 \\
%\mu\partial_2u_3 \\
%\lambda(\partial_1u_1+\partial_2u_2)
%\end{array}\right)\\
%&=-{\rm i}(T(x)D_s u +R^T(x,D_y)u).
%\end{split}
%\end{equation}
%Of course, in the boundary normal coordinates, this expression for $\Lambda(s)$ has to be put into a local tensorial form.
%\par
For a given $s\ (0<s\ll 1)$, let $\mathcal{T}$ be given such that $0<\mathcal{T}-s\ll 1$. The parametrix $U=U(y',x^3;h)$ to the boundary value problem (\ref{transformed eq}) satisfies locally
\[\begin{split}
&\mathcal{M}U\sim 0~~\mathrm{mod}~\mathcal{O}(h^\infty\mathcal{S}(2))\text{ in }\mathbb{R}^2\times[s,\mathcal{T}],\\
&U|_{x^3=\mathcal{T}}=I.
\end{split}\]

Equation (\ref{spec}) implies that the solution operator of the factor
$(hD_s)-S^*(x,hD_{y'};h)+hK(x,hD_{y'};h)$ in the factorization
(\ref{decomp}), for decreasing $s$, is decaying of
order $\mathcal{O}(h^{\infty})$. Hence,
$U$ satisfies
\[
   ((hD_s) - S(x,hD_{y'};h)) U \sim 0~~\mathrm{mod}~\mathcal{O}(h^\infty\mathcal{S}(1)).
\]
Thus,
\[
   (hD_s) U|_{x^3=s}\sim S(x,hD_{y'};h) U|_{x^3=s}
         ~~\mathrm{mod}~\mathcal{O}(h^\infty\mathcal{S}(1)).
\]
Therefore,

\begin{equation}\label{Lambda}
   \Lambda^h(s)(y',hD_{y'};h) \sim
   -{\rm i} (D(x) S(x,hD_{y'};h) + R^T(x,hD_{y'}))|_{x^3=s} ~~\mathrm{mod}~\mathcal{O}(h^\infty\mathcal{S}(1)).
\end{equation}
Then formula (\ref{princ}) follows immediately.
\end{proof}\\

Next we establish the relation between the principal symbol
$\lambda_0(s)$ and the lower order ones $\lambda_{j}(s)$,
$j\leq -1$. Below, we use the notation mod $(T_s^k,h\mathcal{S}(1))$ to
indicate ignoring terms in $h\mathcal{S}(1)$ and in $T_s^k=\{\text{symbol}\ p_s(y',\eta')$ which depends only on the $s$-derivatives of $\lambda(y',s),\mu(y',s),\rho(y',s)$ up to order $k\}$.

\medskip
\noindent
\begin{proposition}
There is a bijective linear map $W(y',s,\eta')$ on the set of
$3\times 3$ matrices which depends only on $\lambda,\mu,\rho$, but not
on their normal derivatives, such that
\begin{equation}
   \lambda_{j}(s)(y',\eta')\sim W(\cdot,s,\cdot)(D_{s} \lambda_{j+1}(s))(y',\eta'))
  ~~\operatorname{mod}~(T_s^{-(1+j)},h\mathcal{S}(1)) ,
\end{equation}
for any $j\leq -1$.
\end{proposition}

\begin{proof}
Throughout the proof, we read $x^3 = s$. First, we note that for
$j \leq -1$
\[
   \lambda_{j}(s)(y',\eta')
       = -{\rm i} D(x) S_{j}(x,\eta') .
\]
From (\ref{order 1}), we obtain
\[S_0S_{-1}+S_{-1}S_0+D^{-1}(R+R^T)S_{-1}=-D_sS_0-D^{-1}F_0S_0-D^{-1}F_1~~\text{mod}~(T_s^0,h\mathcal{S}(1)).\]
	Moreover,
	\[F_0=D_sD~~\text{mod}~(T_s^0,h\mathcal{S}(1))\]
	and
	\[F_1=D_sR^T~~\text{mod}~(T_s^0,h\mathcal{S}(1))\ .\]
	Hence,
	\[\begin{split}S_0S_{-1}+S_{-1}S_0+D^{-1}(R+R^T)S_{-1}&=-D_sS_0-D^{-1}(D_sD)S_0-D^{-1}(D_sR^T)\\
	&=-D^{-1}D_s(DS_1+R^T)\\
	&={\rm i}D^{-1}D_s\lambda_1(s)~~\text{mod}~(T_s^0,h\mathcal{S}(1)).
	\end{split}\]
	Following the proof of Lemma \ref{lemma1}, we find that $\lambda_{-1}(s)$ satisfies
	\[(Q+\rho G)S_0^{-1}D^{-1}\lambda_{-1}(s)-\lambda_{-1}(s)S_0=D_s\lambda_0(s)~~\text{mod}~(T_s^0,h\mathcal{S}(1)).\]
	We note that
\[\operatorname{Spec}(S_0)\subset\mathbb{C}_+,
   ~~~\operatorname{Spec}(-(Q+\rho G)S_0^{-1}D^{-1})\subset\mathbb{C}_+.\]
Hence, defining $W(x,\eta')(Y)$ as the solution $X$ of
\[
   (Q+\rho G)S_1^{-1}D^{-1}X-X S_0=Y ,
\]
we obtain
\[
   \lambda_{-1}(s)(y',\eta') = W(\cdot,s,\cdot) D_s \lambda_0(s)(y',\eta')~~\operatorname{mod}~(T_s^0,h\mathcal{S}(1)).\]
For $j \leq -2$, $S_{j}$ contains $s$-derivatives of $\lambda,\,\mu,\,\rho$ up to order $-j$. Then, inductively, we get
	\[(Q+\rho G)S_0^{-1}D^{-1}\lambda_{j}(s)-\lambda_{j}(s)S_0=D_s\lambda_{j+1}(s)~~\operatorname{mod}~(T_s^{-(1+j)},h\mathcal{S}(1)) .
\]
Thus, we have proved the claim.
\end{proof}\\

We conclude this section by presenting the Riccati equation that
$\Lambda(s)$ satisfies:\\

\begin{corollary}
Define
\begin{equation}\label{modified DNmap}
   \hat{\Lambda}(s)={\rm i} D^{-1} \Lambda^h(s) ;
\end{equation}
$\hat{\Lambda}(s)$ satisfies, $\mathrm{mod}$ $\mathcal{O}(h^\infty\mathcal{S}(1))$, the Riccati
equation
\begin{equation}\label{Riccati}
   hD_s\hat{\Lambda}(s) + J_1(s) \hat{\Lambda}(s)
       + \hat{\Lambda}(s) K_1(s) + \hat{\Lambda}(s)^2
          + F_2(s) = 0~~~(0 \le s\ll 1) ,
\end{equation}
where
\[
   J_1(s) = D^{-1}(R+hF_0) ,\quad K_1(s) = -D^{-1}R^T ,
\]
and
\[
   F_2(s) = -hD_s (D^{-1} R^T) - D^{-1} (R + R^T + hF_0) D^{-1} R^T
        + D^{-1} (hF_1 + Q + \rho  G) + (D^{-1} R^T)^2 ,
\]
with $x^3 = s$.
\end{corollary}

\begin{proof}
This follows straightforwardly from (\ref{Riccati for S}) and
(\ref{Lambda}).
\end{proof}\\

\medskip

%\begin{rem}
%By modifying $\hat{\Lambda}(s)$ with a pseudodifferential operator in
%$\mathrm{Op\,S^{-\infty}}$, the Riccati type equation holds without any error
%in $\mathrm{Op\,S^{-\infty}}$.
%\end{rem} 

Invoking a forward Euler scheme to solve the Riccati equation, we
obtain an approximate propagation of the boundary data into the
interior of $\Omega$, layer by layer. With the explicit reconstruction
that will be presented in the next section, we obtain formally a
layer-stripping algorithm for our inverse boundary value problem. The Riccati-type equation is expected to be 
highly unstable, especially for high frequency modes \cite{SCII}. So the propagation of DN map
will deteriorate. This reveals the ill-posedness of the problem of recovering the parameter in the interior.

\section{Reconstruction of the Lam\'{e} parameters and density}${}$
\newline
\indent
We present the reconstruction of $(\lambda, \mu, \rho)$, as well
as all their derivatives at $x^3=s$ from the full symbol of
$\Lambda(s)$. We first consider the principal symbol of operator
$-h^2\sum_{j,k,l=1}^3\nabla_j (C^{ijkl}\varepsilon_{kl}(v))$. By the
transformation rule of tensor, we have
\begin{equation}\label{transformation}
% \left\{
\begin{array}{l}
N := (\displaystyle\sum_{j,l=1}^3 C^{ijkl}\eta_j\eta_l; 1\le i,\, k\,\le 3)=J\dot{N}J^T\\
\qquad\text{with}\,\, \dot{N} = (\displaystyle\sum_{j,l=1}^3 \dot{C}_{ijkl}\xi_j\xi_l; 1\le i,\, k\,\le 3).
\end{array}
% \right.
\end{equation}
For any $x$ near $\Gamma(s)$, we choose a unit vector $n(x) =
(n_1,n_2,n_3)\in\mathbb{R}^3$ depending smoothly on $x$. Then any $\xi = (\xi_1,\xi_2,\xi_3)\in
\mathbb{R}^3$ can be written as $\xi=qn(x)+m(x,\xi)$ for some $q \in \mathbb{R}$
and $(m_1,m_2,m_3) =: m \perp n$. We define $\dot{D} = \dot{D}(x)$,
$\dot{R} = \dot{R}(x,\xi)$, $\dot{Q} = \dot{Q}(x,\xi)$ as follows,
\begin{equation}
% \left\{
\begin{array}{l}
   \dot{D} = (\displaystyle \sum_{j,l=1}^3
       \dot{C}_{ijkl} n_j n_l; 1\le i,\, k\,\le 3) ,
\\
   \dot{R} = (\displaystyle \sum_{j,l=1}^3
       \dot{C}_{ijkl} m_j n_l; 1\le i,\, k\,\le 3) ,
\\
   \dot{Q} = (\displaystyle \sum_{j,l=1}^3
       \dot{C}_{ijkl} m_j m_l; 1\le i,\, k\,\le 3) .
\end{array}
% \right.
\end{equation}
We have (compare with (\ref{mathcal M}))
\begin{equation}
\dot{M}=\dot{D}q^2+{\large(}\dot{R}+(\dot{R})^T{\large)}q+\dot{Q}+\rho.
\end{equation}
and a smooth factorization according to \eqref{fac M}. More precisely,
there exists a unique $\dot{S}_0=\dot{S}_0(x,\xi)$ depending
smoothly on $x\in\overline\Omega$, and homogeneous of degree
one with respect to $\xi$ such that
\begin{equation}
   \dot{M}={\large(}q-(\dot{S}_0)^\ast{\large)}\dot{D}(q-\dot{S}_0) ,\quad
   \operatorname{Spec}(\dot{S}_0)\subset\mathbb{C}_+.
\end{equation}

We let the direction of $n(x)$ be aligned with the $x^3$ axis.  Using
\eqref{transformation} we find that
\begin{equation}\label{fac M b_normal}
M={\large(}q-J(\dot{S}_0{\large)}^\ast J^{-1})(J\dot{D}J^T){\large(}q-(J^T)^{-1}\dot{S}_0J^T{\large)} .
\end{equation}
Since the linear mapping defined by the matrix $(J^T)^{-1}$ preserves
the orthogonality $m(x,\xi)\perp n(x)$ and the length of $n(x)$, we have
$q=\eta_3$.  We also have
$D = J \dot{D} J^T$. Hence, by the
uniqueness of factorization \eqref{fac M},
\begin{equation}
   S_0=(J^T)^{-1} \dot{S}_0 J^T .
\end{equation}
Combining this with \eqref{princ} and the tensorial transformation
$R=J\dot{R}J^T$ of $\dot{R}$, we obtain
\begin{equation}
\lambda_0(s) = -{\rm i}J(\dot{D}\dot{S}_0+\dot{R}^T)J^T.
\end{equation}

%where 
% For boundary normal coordinates $(x^1,x^2,x^3)=(y',y^3)$ we see that $|\mathrm{d}x^3|=1$ and $\mathrm{d}x^i~(i=1,2)$ are perpendicular to $\mathrm{d}x^3$. Hence for $(y',\eta')\in T^*(\Gamma(s))\backslash0$, $\eta/|\eta|$ and $\mathrm{d}x^3$ are orthogonal unit vectors. Then we have the following more explicit form of $Q(x,\eta),R(x,\eta),D(x)$:
%\[Q(x,\eta)=\left(\begin{array}{ccc}
%(\lambda+2\mu)\eta_1^2+\mu\eta^2_2 & (\mu+\lambda)\eta_1\eta_2 & 0\\
%(\mu+\lambda)\eta_1\eta_2 & (\lambda+2\mu)\eta_1^2+\mu\eta^2_2 & 0\\
%0 & 0 & \mu|\eta|^2
%\end{array}\right),\]
%$\[R(x,\eta)=\left(\begin{array}{ccc}
%0 & 0 & \lambda\eta_1\\
%0 & 0 & \lambda\eta_2\\
%\mu\eta_1 & \mu\eta_2 & 0
%\end{array}\right),\]
%\[T(x)=\left(\begin{array}{ccc}
%\mu & 0 & 0\\
%0 & \mu & 0\\
%0 & 0 & \lambda+2\mu
%\end{array}\right).\]
%The explicit decomposition of $M(x,\tau,\eta,\xi)$ defined in (\ref{mathcal M}) in terms of $\lambda,\mu,\rho$ is
%such that $\text{Spec}(\tilde{S}_1):=\{\mbox{eigenvalues of $\tilde{S}_1$}\}\subset{{\Bbb C}}_+:=\{\zeta\in{\Bbb C}:\, \text{Im}\,\zeta>0\}$, See \cite{GLR}. Although it is a burdensome chore, we have obtained the explicit form of $\tilde{S}_1=\tilde{A}+i\tilde{B}$ is as follows.
%\[M(x,\tau,\eta,\xi)=(\xi-S_1^*(x,\tau,\eta))T(x)(\xi-S_1(x,\tau,\eta)),\]
Due to the isotropy of the elasticity tensor, a rotation of
coordinates $(x_1,x_2,x_3)$ does not affect the form
$\dot{D},\,\dot{S}_1,\,\dot{R}$ and. Hence, for any $x$, we just assume $\xi=(\xi',\xi_3)$, and $n(x)=(0,0,1)$. Then $m(x,\xi)=\xi'$, so $\dot{S}_0(x,\xi)$ depends only on $(x,\xi')$,
\[
   \dot{S}_0(x,\xi) = \dot{A}(x,\xi') + {\rm i} \dot{B}(x,\xi')
\]
and
\[
   \dot{D} = \operatorname{diag}(\mu,\mu,\lambda+2\mu) ,\quad
    \dot{A} = \dot{D}^{-1/2} \tilde{A}\dot{D}^{1/2} ,\quad
   \dot{B} = \dot{D}^{-1/2} \tilde{B} \dot{D}^{1/2} ,
\]
in which
\begin{equation}\label{ABform}
   \tilde{A}=P\left(\begin{array}{ccc}
0 & 0 & -\alpha_1\\
0 & 0 & 0\\
-\alpha_2 & 0 & 0
\end{array}\right)P^* ,\quad
\tilde{B}=P\left(\begin{array}{ccc}
a & 0 & 0\\
0 & b & 0\\
0 & 0 & c
\end{array}\right)P^*
\end{equation}
with
\[
P = P(\xi') = \left(\begin{array}{ccc}
\xi_1|\xi'|^{-1} & \xi_2|\xi'|^{-1} & 0\\
\xi_2|\xi'|^{-1} & -\xi_1|\xi'|^{-1} & 0\\
0 & 0 & 1
\end{array}\right) ,
\]
\[\alpha_1=\frac{(\lambda+\mu)|\xi'|}{\sqrt{\mu(\lambda+2\mu)}}\frac{1}{1+\gamma} ,\quad \alpha_2=\gamma\alpha_1 ,\quad b=\sqrt{\frac{\mu|\xi'|^2+\rho}{\mu}},\]
\[
   c = \frac{1}{1+\gamma}\sqrt{(1+\gamma)^2\frac{\mu|\xi'|^2+\rho}{\lambda+2\mu}-\frac{(\lambda+\mu)^2|\xi'|^2}{\mu(\lambda+2\mu)}} ,\quad
   a = \gamma c
\]
and
\[
   \gamma = \sqrt{\frac{((\lambda+2\mu)|\xi'|^2
      + \rho)(\lambda+2\mu)}{\mu(\mu|\xi'|^2+\rho)}} .
\]
We substitute $\xi' = [|\xi'|,0]^T$ and identify $\xi'$ with $|\xi'|$; then, after some calculations, we find that
\[\begin{split}&\lambda_0(s)(y,\eta)\\
=&-{\rm i}J\large(\dot{D}(x)\dot{S}_0(x,\xi')+\dot{R}^T(x,\xi')\large)J^T=J\dot{\Lambda}_0J^T\,\end{split}\]
with
\[\begin{split}
\dot{\Lambda}_1&=(\dot{\lambda}_1^{(ik)} ;\,1\le i,\,k\le 3)\\
&=\left(\begin{array}{ccc}
a\mu & 0 & {\rm i}\alpha_1\sqrt{\mu(\lambda+2\mu)}-{\rm i}\mu|\xi'|\\
0 & b\mu & 0\\
{\rm i}\alpha_2\sqrt{\mu(\lambda+2\mu)}-{\rm i}\lambda|\xi'|& 0 & c(\lambda+2\mu)
\end{array}\right).\end{split}\]
Since $J$ is known and independent of $\lambda,\,\mu,\,\rho$, we only need to
consider $\dot{\Lambda}_1$ for recovering  $\lambda,\,\mu,\,\rho$ and their derivatives at $x^3=s$.

\medskip
\noindent
\noindent\textit{First step.} We will recover $\lambda,\mu,\rho$.
Note that $\,\dot{\lambda}_0^{(22)}(x,\xi') = \sqrt{|\xi'|^2\mu^2+\rho\mu}$. Then we can first get $\mu$ and $\rho$ as follows. Observe that
 \[\dot{\lambda}_0^{(22)}(x,\sqrt{2}c_0^{-1})^2-\dot{\lambda}_0^{(22)}(x,c_0^{-1})^2=c_0^{-2}\mu^2,\] 
for any scaling constant $c_0>0$. Hence we set $c_0=1$ in the rest of this section. Then we find
\[\mu=\sqrt{\dot{\lambda}_0^{(22)}(x,\sqrt{2}c_0^{-1})^2-\dot{\lambda}_1^{(22)}(x,c_0^{-1})^2}\]
and
\[\rho=\frac{1}{\mu}(\dot{\lambda}_0^{(22)}(x,c_0^{-1})^2-\mu^2).\]
For $\lambda$, first notice that we have
\[\frac{\dot{\lambda}_0^{(11)}(x,c_0^{-1})^2}{\dot{\lambda}_0^{(33)}(x,c_0^{-1})^2}=\frac{(\lambda+2\mu+\rho)\mu}{(\mu+\rho)(\lambda+2\mu)}.\]
Since we have already computed $\mu$ and $\rho$, we get
\[\frac{\lambda+2\mu+\rho}{\lambda+2\mu}=1+\frac{\rho}{\lambda+2\mu},\]
and then obtain $\lambda$.\\

\noindent\textit{Second step.} We recover $\partial\lambda,\,\partial\mu,\,\partial\rho$ of $\lambda,\,\mu,\,\rho$, we first note that
\begin{equation}\label{derivative1}2\mu\partial\mu=\partial(\dot{\lambda}_0^{(22)}(x,\sqrt{2}c_0^{-1})^2-\dot{\lambda}_1^{(22)}(x,c_0^{-1})^2)\end{equation}
from which we can recover $\partial\mu$. Then from
\begin{equation}\label{derivative2}\rho\partial\mu+\mu\partial\rho=\partial(\dot{\lambda}_1^{(22)}(x,c_0^{-1})^2-\mu^2),\end{equation}
we recover $\partial\rho$. Finally we recover $\partial\lambda$ from
\begin{equation}\label{derivative3}\partial\lambda+2\partial\mu=\partial\left(\left(\frac{\mu+\rho}{\mu}\frac{\dot{\lambda}_1^{(11)}(x,c_0^{-1})^2}{\dot{\lambda}_1^{(33)}(x,c_0^{-1})^2}-1\right)^{-1}\rho\right).\end{equation}\\

\noindent
\textit{Final step.} We recover higher order derivatives of
$\lambda,\mu,\rho$. Differentiating equations
(\ref{derivative1})-(\ref{derivative3}) $k-1$ times, we obtain linear
equations for $\partial^k\mu,\partial^k\lambda$ and
$\partial^k\rho$. The coefficients for them are the same as those for
$\partial\lambda,\,\partial\mu,\,\partial\rho$ in
(\ref{derivative1})-(\ref{derivative3}). Thus we can recover
$\partial^k\mu,\partial^k\lambda,\partial^k\rho$, using
$\partial^j\dot{\Lambda}_0(s) (y,\eta)$ for $j=1,2,\cdots, k$, and
$\partial^j\mu,\partial^j\lambda,\partial^j\rho$ for $j=1,2,\cdots,
k-1$. So we can recover all the derivatives of $\lambda,\mu,\rho$
recursively.

\section{Further implications}${}$
\newline
\indent
In this section, we show how to get a decomposition into incoming/outgoing waves via the factorization $(\ref{fac M})$. In this section, we take $T=\infty$ and allow $\tau$ to take complex values. We assume $\tau$ is  in the set
\[\Pi_{0}=\{\tau\in\mathbb{C};\Re \tau:=\text{real part of $\tau$}>0\}.\]

For $u\in W((0,\infty);\Omega)$, we introduce the Laplace transform $\mathcal{L}$ for $\tau\in\Pi_0$:
\[(\mathcal{L}u)(x,\tau)= \int_{0}^\infty e^{-\tau t}u(x,t)\mathrm{d}t,\]
where $(\mathcal{L}u)(\cdot,\tau)\in H^2(\Omega)$.
The inverse Laplace transform is given by
\[(\mathcal{L}^{-1}v)(x,t)=\frac{1}{2\pi {\rm i}}\int_{\gamma-{\rm i}\infty}^{\gamma+{\rm i}\infty}e^{\tau t}v(x,\tau)\mathrm{d}\tau,\]
with $\gamma\in\Pi_{0}$.

Applying above Laplace transform to (\ref{EQ no.1}), we get
\[\mathcal{M}v=\rho \hat{\tau}^2g^{ik}v_k-h^2\displaystyle\sum_{j,k,l=1}^3\nabla_j (C^{ijkl}\varepsilon_{kl}(v))=0,\]
with $h=\frac{1}{|\tau|}$, $\hat{\tau}=\frac{\tau}{|\tau|}$. $\mathcal{M}$ can be viewed as a semiclassical pseudodifferetial operator with a small parameter $h=\frac{1}{|\tau|}$

We rewrite above equation up to the leading order terms in the following form
\begin{equation}
hD_s\left(\begin{array}{c}
v\\
hD_sv
\end{array}\right)\sim\left(\begin{array}{cc}
0 & 1\\
-D^{-1}(Q+\rho G\tau^2) & -D^{-1}(R+R^T)
\end{array}\right)\left(\begin{array}{c}
v\\
hD_sv
\end{array}\right)~~\mathrm{mod}~\mathcal{O}(h\mathcal{S}(2)).
\end{equation}

Denote $M(x,\hat{\tau},\eta)$ to be the principal symbol of $\mathcal{M}$, as in $(\ref{fac M})$, we have the factrorization
\begin{equation}\label{Mcomplex}
M(x,\hat{\tau},\eta)=(\eta_3-S_0^-(x,\hat{\tau},\eta'))D(x)(\eta_3-S_0^+(x,\hat{\tau},\eta')),
\end{equation}
for $\tau\in\Pi_0$.
Here, similar to (\ref{ABform}),
\begin{equation}
   S_0^+(x,\hat{\tau},\eta')=(J^T)^{-1}\dot{D}^{-1/2} (\tilde{A}+{\rm i}\tilde{B} )\dot{D}^{1/2}J^T
\end{equation}
and
\begin{equation}
   S_0^-(x,\hat{\tau},\eta')=J\dot{D}^{1/2}( \tilde{A}^T-{\rm i}\tilde{B}^T)\dot{D}^{-1/2}J^{-1},
\end{equation}
where 

\begin{equation}\label{ABformtau}
   \tilde{A}(x,\hat{\tau},\eta')=P\left(\begin{array}{ccc}
0 & 0 & -\alpha_1\\
0 & 0 & 0\\
-\alpha_2 & 0 & 0
\end{array}\right)P^* ,\quad
\tilde{B}(x,\hat{\tau},\eta')=P\left(\begin{array}{ccc}
a & 0 & 0\\
0 & b & 0\\
0 & 0 & c
\end{array}\right)P^*
\end{equation}
with
\[\alpha_1=\frac{(\lambda+\mu)|\xi'|}{\sqrt{\mu(\lambda+2\mu)}}\frac{1}{1+\gamma} ,\quad \alpha_2=\gamma\alpha_1 ,\quad b=\sqrt{\frac{\mu|\xi'|^2+\rho\hat{\tau}^2}{\mu}},\]
\[
   c = \frac{1}{1+\gamma}\sqrt{(1+\gamma)^2\frac{\mu|\xi'|^2+\rho\hat{\tau}^2}{\lambda+2\mu}-\frac{(\lambda+\mu)^2|\xi'|^2}{\mu(\lambda+2\mu)}} ,\quad
   a = \gamma c
\]
and
\[
   \gamma = \sqrt{\frac{((\lambda+2\mu)|\xi'|^2
      + \rho\hat{\tau}^2)(\lambda+2\mu)}{\mu(\mu|\xi'|^2+\rho\hat{\tau}^2)}} .
\]

In all formulas for $a,b,c,\alpha_1,\alpha_2,\gamma$, $\sqrt{z}$ is defined on $\mathbb{C}\setminus(-\infty,0]$ with $\Re\sqrt{z}>0$. Indeed $\alpha_1$, $\gamma$ and $b$ are well defined, and $\Re\gamma>0$. Furthermore, we note that if $\Im\tau>0$, then $\Im\{\frac{\mu|\xi'|^2+\rho\hat{\tau}^2}{\lambda+2\mu}\}>0$ and $\Im\gamma<0$, while $\Im(1+\gamma)^2<0$. Thus $c$ is well defined for $\Im\tau>0$. Similarly, we can verify $c$ is well defined for $\Im\tau<0$.

We now show that
\begin{equation}\label{local1}
\mathrm{Spec}(S_0^+)\subset\mathbb{C}_+,~~~\mathrm{Spec}(S_0^-)\subset\mathbb{C}_-
\end{equation}
for any $\tau\in\Pi_0$. The spectrum of $S_0^+$ unified with the spectrum of $S_0^-$ are the roots of $\det(M(x,\hat{\tau},\eta))$ as a polynomial in $\eta_3$ for $\tau\in\Pi_0$. To obtain statement (\ref{local1}), we only need to show that there are no real roots of $\det(M(x,\hat{\tau},\eta))$ for any $\tau\in\Pi_0$. Then because (\ref{local1}) holds true for positive $\tau$, the eigenvalues of $S_0^+$ or $S_0^-$ cannot intersect the real line. Roots of $\det(M(x,\hat{\tau},\eta))$, which are same as the roots of $\det(\dot{M})$ in $q$ satisfy
\[\mu q^2+\mu|\xi'|^2+\rho\hat{\tau}^2=0\]
or
\[\left((\lambda+2\mu)q^2+\mu|\xi'|^2+\rho\hat{\tau}^2\right)\left(\mu q^2+(\lambda+2\mu)|\xi'|^2+\rho\hat{\tau}^2\right)-(\lambda+\mu)^2q^2|\xi'|^2=0.\]
For $\tau\in(0,\infty)$, we have already concluded that $q$ could not be real. For non-real $\tau$, if there is a real $q$ satisfying the above equations, then
\[\Im\hat{\tau}^2=0,\]
which is not possible. 

Let
\[\begin{split}P%=&\left(\begin{array}{cc}
%S_1^*& -1\\
%-S_1 & 1
%\end{array}
%\right)
%\left(\begin{array}{cc}
%(S_1^*-S_1)^{-1} & 0\\
%(S_1^*-S_1)^{-1}(S_1^*S_1-S_1S_1^*)(S_1^*-S_1)^{-1} & (S_1^*-S_1)^{-1}
%\end{array}
%\right)\\
= \left(\begin{array}{cc}
(S_0^--S_0^+)^{-1}S^-_0 & -(S_0^--S_0^+)^{-1}\\
-(S_0^--S_0^+)^{-1}S_0^+ & (S_0^--S_0^+)^{-1}
\end{array}
\right)
.\end{split}\]
and denote its inverse by $P^{-1}$ which is given by
\[P^{-1}=\left(\begin{array}{cc}
1 & 1\\
S_0^+ & S_0^-
\end{array}
\right).\]
Then we find that
\[\left(\begin{array}{cc}
0 & 1\\
-D^{-1}(Q+\rho G\hat{\tau}^2) & -D^{-1}(R+R^T)
\end{array}\right)= P^{-1}\left(\begin{array}{cc}
S_0^+ & 0\\
0 & S_0^-
\end{array}\right)P+\mathrm{l.o.t.}.\]
It follows that
\[\left(\begin{array}{c}
v_+\\
v_-
\end{array}\right)=P\left(\begin{array}{c}
v\\
hD_sv
\end{array}\right)\in W((0,\infty);\Omega)\]
satisfies
\[(hD_s-S_0^+)v_+= 0 ~~\mathrm{mod}~\mathcal{O}(h\mathcal{S}(1))\]
and 
\[(hD_s-S_0^-)v_-= 0~~\mathrm{mod}~\mathcal{O}(h\mathcal{S}(1)).\]

We can view $u_+=\mathcal{L}^{-1}v_+\,\,(u_-=\mathcal{L}^{-1}v_-)$ as representing incoming (outgoing) waves. This identification is justified by noticing that, for example, $v_+(\cdot,\tau)=(\mathcal{L}u_+)(\cdot,\tau)$ is exponentially decaying with increasing $s$ for $\Re\tau>0$. 

We emphasize that the transformed DN map $\Lambda$, here, is different from the one introduced before, that is,
\[\Lambda^\tau=h\mathcal{L}\Lambda_T\mathcal{L}^{-1}.\]
With the relation between the normal directive $D_s$ and the DN map $\Lambda$,
\[hD_s={\rm i}D^{-1}(\Lambda^\tau-R^T)~~\mathrm{mod}~\mathcal{O}(h\mathcal{S}(1)),\]
 we have
\[\left(\begin{array}{c}
v_+\\
v_-
\end{array}\right)= P\left(\begin{array}{c}
v\\
{\rm i}D^{-1}(\Lambda^\tau-R^T)v.
\end{array}\right)~~\mathrm{mod}~\mathcal{O}(h\mathcal{S}(1))\]
on boundary. 
We note that
\[P=(S_0^--S_0^+)^{-1}\left(\begin{array}{cc}
1 & 0\\
0 & -1
\end{array}
\right)P^{-\star}\left(\begin{array}{cc}
0 & -1\\
1 & 0
\end{array}
\right) \]
with
\[P^{-\star}=\left(\begin{array}{cc}
1 & S_0^-\\
1 & S_0^+
\end{array}
\right),\]
where $P^{-\star}=P^{-*}$ for real $\tau$.

After we have identified the elastic parameters on the boundary, we will then know $R$ in $(\ref{QRD})$. Then we can have a decomposition into incoming and outgoing wave constituents ($u^+$ and $u^-$ respectively) on the boundary using $\Lambda_T$ and $R$. Seismic imaging (inverse scattering), array receiver functions, and tomography (also using free-surface multiple scattering) all rely on this decomposition. Discussion of seismic migration and inversion schemes based on this decomposition can be found in \cite{dHdH, K}.
\\

{\bf Acknowledgement}
We thank the anonymous reviewers for very useful comments to improve this paper.

\end{document}